\documentclass[reqno]{article}
\usepackage{amssymb, amsmath, amsthm, amsfonts, amscd}
\usepackage[english]{babel}

\usepackage{graphicx}

\usepackage[margin=3.cm]{geometry}
\usepackage{epstopdf}

\usepackage{breqn}
\usepackage{mathtools}
\usepackage{color}
\usepackage{hyperref}
\hypersetup{
    colorlinks=true,
    citecolor=red,
    filecolor=black,
    linkcolor=blue,
    urlcolor=black
}

\chardef\bslash=`\\ 

    \newtheorem{theorem}{Theorem}[section]
    \newtheorem{corollary}[theorem]{Corollary}
\newtheorem{lemma}[theorem]{Lemma}
\newtheorem{proposition}[theorem]{Proposition}

\newcommand{\N}{\mathbb{N}}

\newcommand{\Z}{\mathbb{Z}}

\newcommand{\T}{\mathbb{T}}

\newcommand{\Id}{\mathrm{Id}}

\def\a{\alpha }

\def\sm{C^{\infty} }

\def\s{\sigma}
\def\t{\tau}

\def\e{\varepsilon}
\def\f{\varphi}
\def\.{\cdot }
\def\ra{\rightarrow}

\def\begeq{\begin{equation*}}
\def\endeq{\end{equation*}}

\title{
\textsc{\textbf{Furstenberg counterexamples over Diophantine rotations}}\\
\author{Nikolaos Karaliolios}
}

\begin{document}

\maketitle

\begin{abstract}
We construct cocycles in $\T \times SU(2)$ over Diophantine rotations that are minimal
and not uniquely ergodic. Such cocycles are dense in an open subset of cocycles over
the fixed Diophantine rotation. By a standard 
argument, they are dense in the whole set of
such cocycles if the rotation satisfies a full-measure
arithmetic condition.
\end{abstract}  

\tableofcontents

\section{Introduction}

In \cite{Furst61}, Furstenberg provided the first examples of diffeomorphisms that are
not uniquely ergodic, but are topologically minimal. These examples are cocycles in
$\T \times \T$, of the form
\begin{equation*}
(x,y) \mapsto (x+\a , y+ \beta + \f (x))
\end{equation*}
where $\a$ is a Liouville rotation and $\int \f = 0$.

The title of the article is admittedly slightly provocative, since it is known that
Furstenberg's construction does not work in the Abelian setting if $\a$, the rotation in
the basis, is Diophantine. In order to make the construction possible, we thus need the
fiber to be a non-Abelian Lie group, in our case the simplest one, $SU(2)$. Moreover,
by the hypoellipticity property of fibered rotation vectors that are Diophantine with
respect to the rotation in the basis, c.f. \cite{NKRotVec}, the fibered rotation vector
needs to be Liouville with respect to the rotation in the basis. This is in contrast with
the construction of \cite{Furst61}, where the fibered rotation number $\beta$
($\beta = 0$ in the original construction) is irrelevant and
the only relevant property of the dynamics in the fibers is
that the non-linearity $\f$ be the coboundary
of a measurable and not any more regular function.

Our construction is remarkably easy modulo the KAM normal form for almost reducible cocycles, c.f.
\cite{NKInvDist}. The KAM normal form was exploited by the author in \cite{NKInvDist},
\cite{NKContSpec} and \cite{NKRotVec}, which pushed the results of \cite{El2002a}, as
a means of simplifying the fibered Anosov-Katok construction used in the proof of
the existence of cocycles exhibiting the dynamical properties that are the subject of
each of the articles.

Such constructions in general become easier when the rotation in the basis is allowed to
vary, imitating the classical Anosov-Katok construction, \cite{AnKat1970}. This, however,
leads to the constructions of examples over a generic rotation, which should be
expected to be Liouville, as the construction proceeds by periodic approximation.
Our construction proceeds by resonant approximation, i.e. by adding tailored resonant
modes to a cocycle reducible to a resonant constant (a constant whose eigenvalues are
integer multiples of $\a$), while keeping $\a$ fixed.
This is the direct fiber-wise analogue of periodic
approximation when the rotation in the basis is kept fixed.

The dynamics of a cocyle $(\a, A(\.)) \in SW^{\infty}(\T, SU(2)$, where
$A(\.) \in C^{\infty}(\T, SU(2))$, is given by
\begin{equation} \label{eqdef dynamics}
(\a, A(\.)).(x,S) \mapsto (x+\a, A(x).S)
\end{equation}
for any $(x,S) \in \T \times SU(2)$. Conjugation of cocycles is fibered congugation,
just like in Furstenberg's construction for cocycles in $\T \times \T$, acting by
\begeq
(0,B(\.)) \circ (\a, A(\.))\circ (0,B(\.))^{-1}  = (\a,B(\.+\a)A(\.)B^*(\.)).
\endeq
A cocycle is reducible if it can be conjugated (via a conjugation of regularity
to be specified in each context) to a constant cocycle. It is called Almost Reducible
if it can be smoothly conjugated arbitrarily close to constant cocycles. By the
theory developed in \cite{KrikAst} and the author in his thesis, \cite{NKPhD},
Almost Reducible cocycles form an open set in the total function space under a relevant full-measure arithmetic
condition on $\a$, somewhat stricter than a classical
Diophantine one.

In the present article, we will focus on the perturbative
scenario, where the mapping $A(\.) $ is close to a constant, and can thus be written
in the form $A(\.) = A\exp(F(\.)) $ for some smooth mapping $F(\.):\T \ra su(2)$, where
$su(2) $ is the Lie algebra of $SU(2)$.

The entire \S \ref{secexistence} of the paper is devoted to the proof of the following
theorem.

\begin{theorem} \label{main theorem}
For any fixed Diophantine rotation $\a$, minimal and non-uniquely ergodic cocycles are
dense in the open set of Almost Reducible smooth cocycles over $\a$ in $\T \times SU(2)$.
\end{theorem}

In order to keep the note short, we will use consistently the notation of the
previous works by the author, \cite{NKInvDist}, \cite{NKContSpec} and \cite{NKRotVec},
which the present article builds upon.

\section{Existence of minimal cocycles} \label{secexistence}

Proving that a cocycle is minimal under the relevant conditions boils down to proving
the following propositions.

First of all, we only need to prove transitivity, thanks to the following easy
proposition.

\begin{proposition}
Let $(\a , A(\. )) \in SW^{0}_{\a} (\T, G)$ be transitive. Then, $(\a , A(\. ))$
is minimal.
\end{proposition}

The proof uses the fact that cocycles are fiber-wise isometries, as can be seen
directly from the definition of the dynamics in eq. \eqref{eqdef dynamics}, and
the minimality of the rotation in the basis.

The following corollary is immediate.
\begin{corollary}
In order to establish minimality, we only need to show that the orbit of $(0,\Id )
\in \T\times SU(2)$ is dense.
\end{corollary}

The strategy of the construction is given in the following proposition, where we
consider an almost reducible cocycle over $\a \in DC$, given directly in KAM normal
form, and follow the notation of \cite{NKRotVec}. Let us establish some
notation before stating the proposition.

Let $(\a, A^{F(\.)})$ be a cocycle in KAM normal form and suppose that $A$ is diagonal.
The cocycle is conjugate to $(\a, \{e^{2i\pi \a _\infty} ,0\})$ via\footnote{The notation
$\{z,w\}$ for matrices in $SU(2)$ or $su(2)$ amounts to keeping the first line of the
matrix, and is consistent with the previous articles by the author, which the
present work is based upon.} a distribution-valued
transfer function by \cite{NKRotVec}. Suppose, now, return times to $0$ for the rotation
by $\a$, i.e.  $n_l \in \N $ such that 
\begin{equation} \label{eqreturn to 0}
n_l \a \ra 0 \in \T,
\end{equation}
that are not return times to $\pm \Id$ for powers of $A _\infty$, i.e.
\begin{equation} \label{eqnonreturn to 0}
\liminf | n_l \a _\infty | _{\bmod \Z /2} \geq \e>0
\end{equation}
for some fixed $\e$.
We recall, cf. \cite{NKInvDist}, the existence of 
\begin{equation*} \label{eqdef conjugations}
G_i(\.) = B_i (\.) D_i B_i^{*}(\.),
\end{equation*}
the congugations that reduce the successive resonant terms of the KAM normal form,
where $B_i (\. ) = \{ \exp(2i\pi k_i \.) ,0\} $ and
\begin{equation} \label{eqdef Di}
D_i= \{\cos \theta _i/2, e^{2i\pi \f _i} \sin \theta _i/2\}
\end{equation}
with $\theta _i $ and
$\f _i$ the angle and phase defining the KAM normal form (see below). This conjugation gives
\begin{equation*}
G_i (\. +\a ) A_{i}e^{F_{i}(\.)} G^{*}_i(\.) = A_{i+1}e^{F_{i+1}(\.)}
\end{equation*}
with $A_{i} = \{e^{2i\pi a _i},0\}$ a diagonal matrix,
$A_{i}\ra \{e^{2i\pi \a _\infty},0\}$ and $F_i(\.) = O(|k_i|^{-\infty})$ in $\sm$. The
angle $\theta _i$ of the KAM normal form is given by
\begin{equation*}
\tan \theta _i = \frac{|\hat{F}_i (k_i)|}{a_i - k_i \a \mod \Z}
\end{equation*}
and the phase $\f _i$ by
\begin{equation*}
e^{2i\pi \f _i } = \frac{\hat{F}_i (k_i)}{|\hat{F}_i (k_i)|}
\end{equation*}
The KAM normal form is characterized by the fact that the only active Fourier mode of
$F_i(\.)$ up to some growing power of $k_i$ is $k_i$ itself, which implies that
\begin{equation*}
F_i(\.) = \{0, \hat{F}_i (k_i)e^{2i\pi k_i \.} \} + O(|k_i|^{-\infty})
\end{equation*}


Calling $H_i = \prod _{1}^{i}G_j$, and pointing out that
$H_{i_l}(0) = \prod _{1}^{i_l}D_i $, suppose that for the $n_l$ as above, there exists
$i_l \ra \infty$ such that
\begin{equation} \label{eqcvgc conj}
H_{i_l}( n_l \a)\left(\prod _{1}^{i_l}D_i \right)^{*} \ra \Id
\end{equation}
and
\begin{equation*}
n_l \|F_{i_l+1}\|_0 \ra 0 
\end{equation*}
Suppose, finally, that the sequence of partial products $\prod _{1}^{j}D_i$ has at least
two accumulation points in $SO(3) \equiv SU(2) \bmod \pm \Id$ along the subsequence
$n_l$.

\begin{proposition} \label{prop transitivity}
Under the conditions outlined here above, the cocycle 
$(\a, A^{F(\.)}) \in SW^{\infty}(\T, SU(2))$ is transitive.
\end{proposition}

The conditions referred to in the statement of the proposition are those of eqs.
\eqref{eqreturn to 0}, \eqref{eqnonreturn to 0} and \eqref{eqcvgc conj},
and the non-convergence of the products $\prod _{1}^{j}D_i$, where the matrices
$D_i$ are defined in eq. \eqref{eqdef Di}.

\begin{proof}
We immediately get that the $n_l$-th iterate of the cocycle satisfies
\begin{equation*}
H_{i_l} (\. + n_l \a ) \left( Ae^{F(\.)} \right)^{(n_l)} H^{*}_{i_l}(\.) =
\left( A_{i_l+1}  e^{F_{i_l+1}(\.)} \right)^{(n_l)}
\end{equation*}
or, equivalently
\begin{equation*}
\left( A_{i}e^{F_{i}(\.)} \right)^{(n_l)}
=
H_i^{*} (\. + n_l \a )  \left( A_{i_l+1}  e^{F_{i_l+1}(\.)} \right)^{(n_l)} H_i(\.) 
\end{equation*}
Calculating the orbit of $(0,\Id)$ gives
\begin{equation*}
\begin{array}{r@{}l}
\left( Ae^{F(\.)} \right)^{(n_l)}(0)
&=
H_{i_l}^{*} (\.+n_l \a )  \left( A_{i_l+1}  e^{F_{i_l+1}(\.)} \right)^{(n_l)}
H_{i_l}(\.)(0) 
\\
&=
H_{i_l}^{*} (n_l \a )  \left( A_{i_l+1}  e^{F_{i_l+1}(\.)} \right)^{(n_l)}(0)
H_{i_l}(0) 
\\
&=
H_{i_l}^{*} (n_l \a )  \left( A_{i_l+1}  e^{F_{i_l+1}(\.)} \right)^{(n_l)}(0)
\prod _{1}^{i_l}D_j 
\\
&\approx
\left( \prod _{1}^{i_l}D_j \right) ^{*}  \left( A_{i_l+1}  e^{F_{i_l+1}(\.)} \right)^{(n_l)}(0)  \prod _{1}^{i_l}D_j 
\\
&\approx
\left( \prod _{1}^{i_l}D_j \right) ^{*}  A_{i_l+1}^{n_l}  \prod _{1}^{i_l}D_j  
\end{array}
\end{equation*}
Since $n_l$ is not a return time to $0$ for $a _\infty$, $ A_{i+1}^{i_l}$ does not
accumulate to $\pm \Id$. This is because
$|a_{i_l+1} - a _\infty| < 2  \|F_{i_l+1}\|_0 $, see the construction of the KAM normal
form in \cite{NKInvDist}. Consequently,
\begin{equation*}
\left( A_{i_l+1}e^{F_{i_l+1}(\.)} \right)^{(n_l)}(0)
\approx
\left( \prod _{1}^{i_l}D_i \right) ^{*}  A_{\infty}^{n_l}  \prod _{1}^{i_l}D_i  
\end{equation*}

Summing up, in the points of accumulation of
$\left( A_{i_l+1}e^{F_{i_l+1}(\.)} \right)^{(n_l)}(0)$, there exist at least two
topological generators of $SU(2)$, since $ A_{\infty}^{n_l} \not \ra \pm \Id$,
$a_{\infty}$ is a minimal rotation in $\mathbb{S}^{1}$ since it is Liouville with
respect to the Diophantine rotation $\a$, and the sequence of algebraic conjugations
$\prod _{1}^{i_l}D_j $ contains at least two distinct accumulation points.
\end{proof}

In a nutshell, the construction consists in considering return times to $0$ for the
rotation in the basis, i.e. $n_l \in \N $ such that $n_l \a \ra 0 \in \T$
along which the conjugations have at least two points of accumulation in
$SU(2) \mod \pm \Id \equiv SO(3)$. If these return times are not return times to $0$
for the rotation corresponding to the fibered rotation number, the result is proved,
since we have obtained that any invariant set containing $(0,\Id)$ also contains two
topological generators of $SU(2)$ in its fiber above $0 \in \T$. The invariant
subset thus contains $\T \times SU(2)$, and transitivity follows, as does minimality.
The crux of the argument thus boils down to finding conditions allowing for the
existence of such a sequence $n_l$ that is compatible with measurable reducibility,
and actually constructing it.

This is done in the next proposition.

\begin{proposition} \label{prop construction}
Let $\a \in DC$. Let, also, $(\a , A(\. ))$ be an almost reducible cocycle, given directly in KAM normal form,
whose angle and phase parameters satisfy
\begin{equation*}
\begin{array}{r@{}l}
\theta &\in \ell ^{2}
\\
\sum \theta _{i} &\text{ diverges}
\\
\f  &\equiv 0
\end{array}
\end{equation*}
and call $k_i$ the resonances of the KAM normal form of the cocycle. Then, the sequence of iterates
\begin{equation*}
n_i = k_i - k_{i-1}
\end{equation*}
satisfies the hypotheses of prop. \ref{prop transitivity} for the full sequence $i$
of the resonances of the KAM normal form.
\end{proposition}

A short lemma will be useful.
\begin{lemma} \label{lemma modification}
The resonances of the KAM normal form of the cocycle satisfy
\begin{equation} \label{eqcond iterates}
|k_i(k_i - k_{i-1})\a |_{\bmod \Z /2} \geq \e
\end{equation}
for some fixed $\e>0$.
\end{lemma}

\begin{proof}
    It can be easily seen that the condition of eq. \eqref{eqcond iterates} is satisfied
by the cocycle as follows. Fix a sequence in
$(\theta_i )\in \ell ^{2}$. Then, at each step the resonant mode and the argument of
the matrix $A_i$ need to satisfy, by \cite{NKInvDist},
\begin{equation*}
\sqrt{(a_i - 2i\pi k_i \a)^2 + |\hat{F}_i(k_i)|^2} = |k_{i+1}\a|_{\bmod \Z /2}
\end{equation*}
where $k_{i+1}$ only subject to $k_{i+1} > k_i^{s_i}$ with $s_i \ra \infty$, and
\begin{equation*}
 \frac{|\hat{F}_i(k_i)|}{a_i - 2i\pi k_i \a} = \tan \theta _i
\end{equation*}
while $|\hat{F}_i(k_i)| = O(k_i^{-\infty})$. There are enough free parameters so that
the choice of eq. \eqref{eqcond iterates} be possible. We remind that $a_\infty$ is
a free parameter of the construction and that it satisfies
\begin{equation*}
a_\infty - k_i \a = O(|k_i^{-\infty}|)
\end{equation*}

In particular, disregarding the diophantine constant of $\a$ and some irrelevant
constants coming from asymptotics of the type $O(|k_i|^{-\infty})$ for simplicity, it
holds that
\begin{equation} \label{eqcvgc to 0}
|n_i \a | \geq |k_i|^{-\t}
\end{equation}
while, by the construction of the cocycle,
\begin{equation*}
|n_i \a | = O( |k_{i-1}|^{-\infty})
\end{equation*}
For the condition of eq. \eqref{eqcond iterates} to be satisfied, we need that
\begin{equation} \label{eqbounded from 0}
|n_i a_i |_{\bmod \Z /2} \approx |n_i k_i \a |_{\bmod \Z /2}
\end{equation}
be bounded away from $0$. If this is not the case already for $n_i$ as defined,
and $|n_i a_i | \ra 0$, using the Diophantine property of $\a$, we get that
\begin{equation*}
 |n_i k_i \a |_{\bmod \Z /2} = |k_i| |n_i \a |_{\bmod \Z /2}  \geq |k_i|^{-\t+1}
\end{equation*}
which holds because $ |n_i \a |_{\bmod \Z /2}$ is small. We can still have the
condition of eq. \eqref{eqcvgc to 0} together with the one of eq.
\eqref{eqbounded from 0} by choosing $n_i$ as
\begin{equation*}
k_i^{\t_i}( k_i - k_{i-1})
\end{equation*}
for some appropriate $0 \leq \t_i \leq \t -1$.
\end{proof}

We will ignore the modification introduced in the proof
of the lemma for simplicity in the proof of the proposition,
which we now present.

\begin{proof}
Firstly, the hypotheses on the summability properties of $\theta$ imply that the
cocycle is measurably reducible (by $\ell ^{2}$ summability, see \cite{NKInvDist}),
but not any more regular (the series $\theta$ is not summable, thus the series
$|k_i|^s \theta_i^2$ is not summable for any $s>0$). Consequently, the cocycle is
non-uniquely ergodic, and the transfer function is no more Sobolev regular than
measurable.

Since the angles $\theta$ are not summable, the sequence of partial sums
$\sum _{0}^{i}\theta _j$ has at least two accumulation points in $\mathbb{S}^{1}$. As a
result, the corresponding products $\prod _{1}^{i_l}D_j $ with $D_i$ as defined after
eq. \eqref{eqdef conjugations} have also at least two accumulation points. It is noted
that, since $\f_i \equiv 0$, the $D_i$ form a one-parameter subgroup of $SU(2)$
parametrized by the angle $\theta$.

Moreover, because of hypoellipticity, \cite{NKRotVec}, the rotation number of the
cocycle $a _\infty$ is Liouville with respect to $\a$, and the sequence of resonances
$k_i$ is actually a sequence of good approximations:
\begin{equation*}
|a _\infty - k_i \a |_{\Z /2} = O(|k_i|^{-\infty})
\end{equation*}
which implies that
\begin{equation*}
| (k_i - k_{i-1}) \a |_{\Z} = O(|k_{i-1}|^{-\infty})
\end{equation*}
i.e. $n_i$ are return times to $0$ for the rotation by $\a$.

Furthermore, $n_i \approx k_i$, since $k_i \gg k_{i-1}$. As a consequence,
\begin{equation*}
n_i F_i \ra 0
\end{equation*}
since $F_i =  O(|k_{i}|^{-\infty})$, but $n_i F_{i-1}$ is not guaranteed to converge
to $0$.\footnote{It can in fact be constructed to be so: $|\hat{F}_{i-1}(k_{i-1})| = n_i^{-1}$
is a valid choice for the resonant mode at $i-1$.}

Concerning the quantity
\begin{equation*}
A_{i}^{n_i} = \{ e^{2i\pi n_i k_i \a }, 0 \}
\end{equation*}
it is (or can be made to be) bounded away from $\pm \Id$, cf. the proof of lemma \ref{lemma modification}.
\end{proof}

The proof of the main theorem, thm. \ref{main theorem}, follows easily.

\begin{proof}[Proof of theorem \ref{main theorem}]
The construction is based on the asymptotic dynamical behavior of the cocycle.
Therefore, by modifying the normal form of any given cocycle starting from an
arbitrary high order of resonances, the construction of
proposition \ref{prop construction} can be applied, producing
an arbitrarily small perturbation of the given cocycle.
The cocycle has the desired dynamical
behavior by proposition \ref{prop transitivity}.
\end{proof}

\section{The optimal condition}

The optimal condition is not known, and is related with the accumulation points
of products of matrices of the type
\begin{equation*}
D_i =
\begin{pmatrix}
\cos \theta _i/2 & e^{2i\pi \f _i} \sin \theta _i/2
\\
- e^{-2i\pi \f _i} \sin \theta _i/2 & \cos \theta _i/2
\end{pmatrix}
\end{equation*}
under the condition that $\theta \in \ell ^{2}$.

In our construction we chose $\f _ i \equiv 0$ because this condition defines a one-parameter
subgroup of $SU(2)$ and the divergence of such products is particularly easy to
analyze. However, most choices of diverging sequences for the angle and
the phase should lead to such products non-converging. We think that minimal and
non-uniquely ergodic coycles are abundant within measurably reducible cocycles, which
form an $F_{\s}$-dense set in the total function space by \cite{NKContSpec}.
We actually think that measurable but not $C^{0}$ reducibility should
imply minimality.

In any case, the non-abelian construction makes it possible
to obtain Furstenberg counterexamples over Diophantine
rotations. The construction requires us to focus on the
non-abelian part of the dynamics, since an abelian cocycle
is automatically a smooth coboundary, and the KAM normal
form allows us to eliminate the abelian part of a non-abelian
cocycle. Then, the arithmetics of the fibered rotation
number come into play, which is in contrast with the abelian
construction. The construction is concluded by imposing a
delicate condition on the quantity $\theta $, which controls
the reducibility properties of the cocycle. The condition
is on the boundary of Unique Ergodicity, and moreover
assures that the cocycle is measurably but not $C^{0}$
reducible, which leaves the window open for minimality of
the dynamics, which is then established.

\bibliography{nikosbib}
\bibliographystyle{aomalpha}

\end{document}